\documentclass{amsart}
\usepackage{Preamble}
\usepackage[textsize=scriptsize]{todonotes}

\title{Equivariantly Slice Knots in Symmetric 4-Manifolds}
\author{Malcolm Gabbard}
\begin{document}
\maketitle

\begin{abstract}
    We study the equivariant 4-genus of strongly invertible knots in the $S^3$ boundary of 4-manifolds with involution. We provide techniques for constructing slice disks for knots in various symmetric 4-manifolds via an equivariant version of Marengon and Mihajlovi\`c's tubing construction. Using these techniques, we show that this equivariant 4-genus can differ from the standard 4-genus function of the 4-manifold as well as the equivariant 4-genus of $S^4$. As an example, we show that $S^2\times S^2$ admits an involution such that the figure $8$ knot is equivariantly slice with respect to one of its two strong inversions but not the other.
\end{abstract}

\section{Introduction}

For a smooth 4-manifold $X$ and knot $K\subset S^3$, the $X$-genus of $K$, denoted $g_X(K)$, is the minimal genus of a properly embedded smooth surface $\Sigma$ in $X^\circ=X\backslash B^4$ bounded by $K$. If a knot $K$ has $g_X(K)=0$, $K$ is said to be $X$-slice. The set of all $X$-slice knots is denoted by $\mathcal{S}(X)$. When $X=S^4$, this yields exactly the classical 4-genus $g_4(K)$. Other well-studied 4-manifolds include $\C P^2 \# \overline{\C P}^2$ and $S^2\times S^2$ in which every knot is slice \cite{norman1969dehn}, $\C P^2$ where not every knot is slice but the genus can differ from the classical genus \cite{yasuhara1992slice}, and most recently the $K3$ surface where it has been shown that every knot with fewer than 21 crossing is slice \cite{marengon2022unknotting}. 

We generalize this notion of $X$-genus to the symmetric setting. In particular, we consider strongly invertible knots and 4-manifolds with particular types of involutions. A \textit{strongly invertible knot} $K$ is a knot in $S^3$ invariant under an orientation preserving involution $\tau$ of $S^3$ which fixes exactly two points of $K$. We will fix a standard involution $\tau$ of $S^3$ with fixed point set $S^1$ in Section \ref{sec: the --Genus}. This choice is unique up to conjugation. Given a closed smooth 4-manifold $X$ with involution $\tau$ and $\Sigma\subset fix(\tau)$, a closed surface, we consider strongly invertible knots in the boundary of $X_{\tau,\Sigma}^\circ=X\backslash B^4$, where the $B^4$ is taken to be a small $\tau$-equivariant neighborhood of a point in $\Sigma$. We then define the $(X,\tau,\Sigma)$-genus $\tilde{g}_{X,\tau,\Sigma}(K)$ of a strongly invertible knot $K$ to be the minimal genus of $\tau$-invariant properly embedded surfaces in $(X,\tau)$ bounded by $K$. We say $K$ is ($X,\tau,\Sigma)$-slice if $\tilde{g}_{X,\tau,\Sigma}(K)=0$ and define $\tilde{\mathcal{S}}(X,\tau,\Sigma)$ to be the set of $(X,\tau,\Sigma)$-slice knots. 

We may relax these definitions by considering the minimal genus over all components of $fix(\tau)$, or going even further and considering the minimum over all involutions $\tau$ of $X$ whose fixed point set is two-dimensional, giving us the $(X,\tau)$- and $\tilde{X}$-genus, respectively. Defining $\tilde{\mathcal{S}}(X,\tau)$ and $\tilde{S}(X)$ to be the set of $(X,\tau)$-slice and $\tilde{X}$-slice knots, we get the following immediate inclusions:

$$\mathcal{S}(X)\supseteq\tilde{\mathcal{S}}(X)\supseteq \tilde{\mathcal{S}}(X,\tau) \supseteq \tilde{\mathcal{S}}(X,\tau,\Sigma). $$

When $X=S^4$, the $\tilde{X}$-genus is exactly the equivariant 4-genus of strongly invertible knots which has been well studied \cite{boyle2022equivariant,dai2023equivariant,di2023new,miller2023strongly}. One important question in this study is whether the equivariant 4-genus of a strongly invertible knot depends on the extension of the involution. That is to say, if $\tau'$ is the standard rotation of $S^4$, is $\tS(S^4)=\tS(S^4,\tau')$?

In hopes of gaining further insight to this question, we study the equivalent question for other 4-manifolds. In particular, we prove a similar statement does not hold for $S^2\times S^2$.

\begin{theorem}\label{thm: different S2 sliceness}
    There exists involutions $\tau_1,\tau_2$ on $S^2\times S^2$ such that neither\\ $\tS(S^2\times S^2,\tau_1)$ nor $\tS(S^2\times S^2, \tau_2)$ is contained in the other. 
\end{theorem}

In particular, this immediately implies the following corollary:

\begin{corollary}
    There exists an involution $\tau$ on $S^2\times S^2$ such that: $$\tilde{\mathcal{S}}(S^2\times S^2,\tau)\subsetneq \tS(S^2\times S^2)$$
\end{corollary}

We prove these results by studying invariant plumbing trees and equivariant unknotting numbers. This method comes from an equivariant modification of Marengon and Mihajlovi{\'c}'s techniques using plumbing trees, with which they showed that every knot with unknotting number less then 21 is slice in $K3$ \cite{marengon2022unknotting}. Their methods utilize information about unknotting numbers of knots and is described in detail in Section \ref{sec: the --Genus}. To move this work to the equivariant setting we use Boyle and Chen's work on equivariant unknotting numbers of strongly invertible knots \cite{nouh2009genera} which divides symmetric crossing changes into three types, type $A$, $B$, and $C$, described in Section \ref{sec: Symmetric Plumbings}. We define symmetric structures on plumbing trees (type $B_+, B_-$ and $C$ plumbing trees), guaranteeing the existence of invariant slice disks, given the following are met by both the knot and the 4-manifold. We will use $\Omega$ to refer to an arbitrary element of $\{B_+,B_-,C\}$.

\begin{theorem} \label{thm: tubing plumbings}
    Let $S=S_1\cup\dots \cup S_n$ be a type $\Omega$ plumbing tree in $(X,\tau)$ with fixed vertex $v\in \Sigma$, a component of $fix(\tau)$. Then a strongly invertible knot $K$ is $(X,\tau,\Sigma)$-slice if it bounds a $\tau$-invariant immersed disk $D$ in $B^4$ satisfying:
    \begin{enumerate}
        \item $D$ has $k\leq n-1$ self-intersections,
        \item $D$ has a single type $\Omega$ self-intersection,
        \item the remaining $n-2$ self intersections of $D$ are type $A$.
    \end{enumerate}
\end{theorem}

In practice, this result creates slice disks for strongly invertible knots in 4-manifolds so long as there is a symmetric plumbing tree of spheres in the 4-manifold which has a structure suitably related to a sequence of symmetric crossing changes in an equivariant unknotting of the knot. In particular, a short argument shows that while $S^2\times S^2$ may not be robust enough for every knot to be equivariantly slice, $3(S^2\times S^2)$ is. 

\begin{corollary}\label{cor: S^2timesS^2}
    There exists an involution $\tau$ on $3(S^2\times S^2)$ such that $\tilde{g}_{X,\tau}(K)=0$ for every strongly invertible knot $K$.
\end{corollary}

\subsection*{Acknowledgments} The author would like to thank Dr. Kristen Hendricks for introducing him to this topic and for numerous helpful conversations as well as feedback on earlier versions of the paper. 
    
\section{The $(X,\tau)$-Genus}\label{sec: the --Genus}

Given a closed smooth 4-manifold $X$, let $X^\circ=X\backslash B^4$. Then the $X$-genus of a knot in $S^3$ is defined as follows:

\begin{definition}
    Given a closed smooth 4-manifold $X$, the \textit{$X$-genus} of a knot $K\subset S^3$, denoted $g_X(K)$, is the minimal genus of a properly embedded smooth surface $\Sigma \subset X^\circ$ with $\partial\Sigma=K$. If $g_X(K)=0$, we say $K$ is $X$-slice and let $\mathcal{S}(X)=\{K\subset S^3\, : \, K \text{ is } X\text{-slice}\}$.
\end{definition}

One can see that the $S^4$-genus of a knot is its 4-genus, i.e. $g_{S^4}(K)=g_4(K)$. Additionally, any knot which is slice in $B^4$ can be isotoped to be slice in a neighborhood of the boundary. This means that $\mathcal{S}(S^4)\subseteq\mathcal{S}(X)$ for all 4-manifolds $X$. It is also well-known that every knot is slice in $S^2\times S^2$ and $\C P^2 \# \overline{\C P}^2$. That is to say, $\{Knots\}=\mathcal{S}(S^2\times S^2)=\mathcal{S}(\C P^2 \# \overline{\C P}^2)$ \cite{norman1969dehn}. There are some 4-manifolds $X$ that have sets of slice knots falling somewhere between, such as $\C P^2$ which has $\mathcal{S}(S^4)\subsetneq\mathcal{S}(\C P^2) \subsetneq \mathcal{S}(S^2\times S^2)$. One interesting 4-manifold in this regard is the $K3$ surface. As mentioned in the introduction, it was shown in Theorem 1.1 of \cite{marengon2022unknotting} that every knot with 21 or fewer crossings is slice in $K3$. However, at this time not much more is known about the landscape of slice knots in $K3$. The techniques of \cite{marengon2022unknotting} will be reviewed in Section \ref{sec: Symmetric Plumbings}, as they serve as a backbone for some of our main results.

Our focus now is to define and provide initial results for an equivariant version of the $X$-genus of strongly invertible knots.

For the remainder of this paper we fix an involution $\tau$ on $S^3$. Specifically, we identify $S^3$ with the one point compactification of $\R^3$, and define $\tau$ to be the extension of the $\pi$-rotation of $\R^3$ about the $z$-axis. We will later refer to involutions on 4-manifolds by $\tau$ as well, understanding that the involution on the 4-manifold restricts to the involution $\tau$ on some specific $S^3$ in the 4-manifold, namely a boundary component of the 4-manifold or the 4-manifold minus a ball.

We know recall the definition of strongly invertible knots, some elementary constructions, and the equivariant 4-genus.

\begin{definition}
    A \textit{strongly invertible knot} is a knot $K\subset S^3$ invariant under $\tau$ such that $K\cap fix(\tau)=\{p_1,p_2\}$, two distinct points.
\end{definition}

\begin{remark}
    It is common in the literature to write $(K,\tau)$ for a strongly invertible knot. However as we have fixed $\tau$, we simply refer to a strongly invertible knot as $K$. It is important, then, to remember that two ambient isotopic classical knots, $K_1$ and $K_2$, might define distinct strongly invertible knots. 
\end{remark}

\noindent We consider strongly invertible knots up to equivariant ambient isotopy. For a strongly invertible knot $K$, we recall a few key constructions and refer the reader to the following paper of Sakuma for a thorough treatment of foundational results for strongly invertible knots \cite{sakuma1986strongly}. We recall how to construct two important classical knots associated with a given strongly invertible knot. Given a strongly invertible knot $K$, we get two \textit{half-axes}, which are the two arcs in $fix(\tau)$ with endpoints on $K$. From these half-axes, we are able to construct two quotient knots associated with $K$. These are the knots $K^i=(K\cup h_i)/ \tau$. See Figure \ref{fig: quotient knots} for the knots $K^i$ associated with the figure eight knot, which are both constructions which will be used consistently in examples and proofs throughout. 

\begin{figure}
    \centering
    \includegraphics[height=0.25\linewidth]{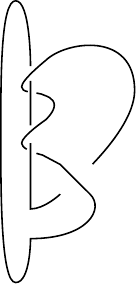}\quad \quad \quad 
    \includegraphics[height=0.25\linewidth]{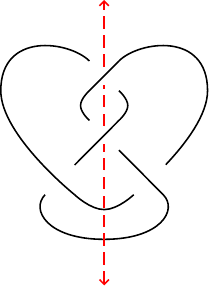}\quad \quad \quad 
    \includegraphics[height=0.25\linewidth]{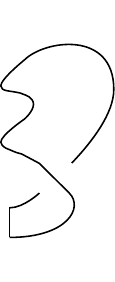}
    \caption{The figure 8 knot $K$ with strong inversion $\tau$ (middle) and its quotients $K^1=T(2,5)$ (left) and $K^2=\text{unknot}$ (right).}
    \label{fig: quotient knots}
\end{figure}

We will make use of equivariant concordance of strongly invertible links in later sections, so we clarify our definition of strongly invertible link here, as it differs somewhat from other sources. Some sources require that each component of the link is, ignoring other components, a strongly invertible knot. We use a less strict definition:

\begin{definition}
    A \textit{strongly invertible link } is an oriented link $L\subset S^3$ invariant under $\tau$, such that each component $L_1$ of $L$ is either a strongly invertible knot or $\tau(L_1)=rL_2$, where $L_2$ is another component of $L$.
\end{definition}

As with strongly invertible knots, we consider strongly invertible links up to equivariant ambient isotopy of $(S^3,\tau)$. For equivariant tubing constructions in later sections, we also need to consider equivariant concordance of strongly invertible links.

\begin{definition}
    Two $n$ component strongly invertible links $L_0$ and $L_1$ are \textit{equivariantly concordant} if there exists a link concordance $f$ from $L_0$ to $L_1$ commuting with $\tau$, the involution on $S^3\times I$, which acts trivially on the $I$ component and as $\tau$ on the $S^3$ component.
\end{definition}

We are now ready to discuss 4-dimensional properties of our strongly invertible knots.

\begin{definition}
    The \textit{equivariant 4-genus} of a strongly invertible knot $K$, denoted $g_{4}(K)$, is the minimal genus of a properly embedded smooth surface $\Sigma\subset B^4$ with boundary $K$ which is invariant under some extension of $\tau$ to $B^4$. If $g_4(K)=0$, we say $K$ is \textit{equivariantly slice}.
\end{definition}

The equivariant 4-genus is well-studied, with many interesting properties and applications \cite{boyle2022equivariant,dai2023equivariant,di2023new,miller2023strongly}. We will discuss a number of properties of this genus as they arise in relation to similar properties in the more general setting we now discuss. 

We now consider closed 4-manifolds $X$ which are equipped with an involution $\tau$ such that $fix(\tau)$ contains some compact surface $\Sigma$. For such 4-manifolds, we may remove a $\tau$-invariant neighborhood of a point in some component $\Sigma$ of the fixed point set. The resulting 4-manifold, which we denote $X^\circ_{\tau,\Sigma}$, is then a 4-manifold with boundary $S^3$ which is invariant under an involution $\tau$. We may then identify the boundary with our standard model of $S^3$ so that $\tau$ restricts to our predetermined action $\tau$ on $S^3$. If $fix(\tau)$ is connected, we will often omit $\Sigma$, instead writing $X^\circ_\tau$. This construction is determined uniquely by $X$, $\tau$, and $\Sigma$, making $X^\circ_{\tau,\Sigma}$ well defined. 
Because of this, we can define the \textit{$(X,\tau,\Sigma)$-genus} of a strongly invertible knot $K$ as follows:

\begin{definition}
    Given a smooth closed 4-manifold $X$ with involution $\tau$ and $\Sigma\subset\text{fix}(\tau)$ a closed surface, the \textit{$(X,\tau,\Sigma)$-genus} of a strongly invertible knot $K\subset S^3$, denoted $\tilde{g}_{X,\tau,\Sigma}(K)$, is the minimal genus of a properly embedded $\tau$-invariant smooth surface $\Sigma_K \subset X^\circ_{\tau,\Sigma}$ with $\partial\Sigma=K$. If $\tilde{g}_{X,\tau,\Sigma}(K)=0$, we say $K$ is $(X,\tau,\Sigma)$-slice and let $\tilde{\mathcal{S}}(X,\tau,\Sigma)=\{\text{Strongly invertible knots }K : \, K \text{ is } (X,\tau,\Sigma)\text{-slice}\}$
\end{definition}

Because $X$ may have multiple involutions and some involutions may have multiple components in the fixed point set, the specification of $\tau$ and $\Sigma$ is necessary. That said, in some situations we may have a connected fixed point set, not care what component we consider, or not care what involution we consider. For this reason, we define two additional variations of the $(X,\tau,\Sigma)$-genus. The first variant ignores choice of component of the fixed point set:

\begin{figure}
    \centering
    \subfigure{\includegraphics[height=0.24\textwidth]{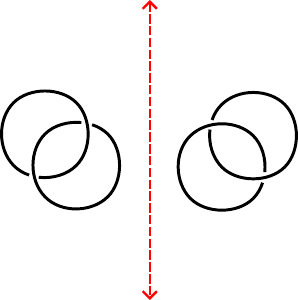}}
    \hspace{4em}\subfigure{\includegraphics[height=0.24\textwidth]{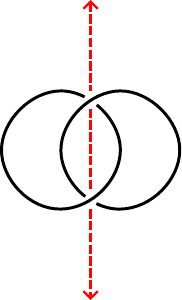}} 
    \subfigure{\includegraphics[height=0.24\textwidth]{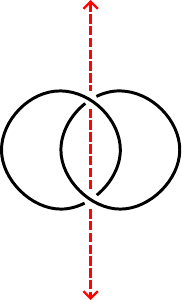}}
    \hspace{6em}\subfigure{\includegraphics[height=0.24\textwidth]{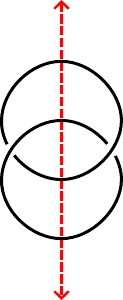}}
    \caption{Associated Hopf links to vertices of type $A$ (left), type $B_+$ (middle left), type $B_-$ (middle right), and type $C$ (right).}
    \label{fig:hopf links}
\end{figure}

\begin{figure}
    \centering
    \subfigure{\includegraphics[height=0.2\textwidth]{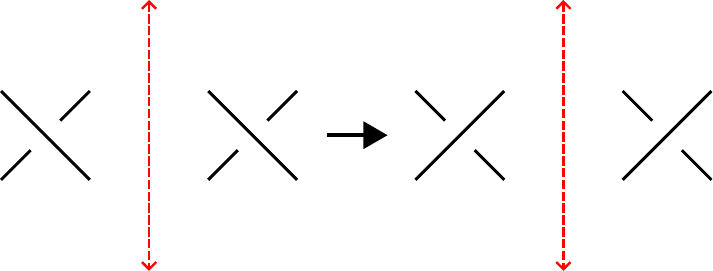}}\\
    \subfigure{\includegraphics[height=0.2\textwidth]{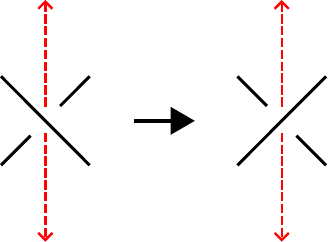}} 
    \hspace{6em}\subfigure{\includegraphics[height=0.2\textwidth]{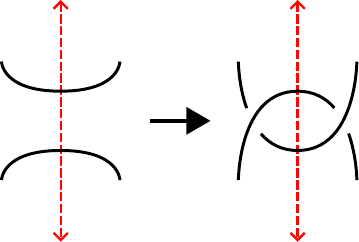}}
    \caption{Crossing changes of type $A$ (top), type $B$ (bottom left), and type $C$ (bottom right)}
    \label{fig:foobar}
\end{figure}

\begin{definition}\label{def: X,tau}
    Given a smooth closed 4-manifold $X$ with involution $\tau$, the \textit{$(X,\tau)$-genus} of a strongly invertible knot $K\subset S^3$, denoted $\tilde{g}_{X,\tau}(K)$, is defined as:
    $$\tilde{g}_{X,\tau}(K) = \text{min}\{\tilde{g}_{X,\tau,\Sigma}(K) \, | \, \Sigma \subset \text{fix}(\tau)\text{ is a closed surface}\}$$
    
\end{definition}

If we wish to ignore choice of involution, we then get the following genus:

\begin{definition}\label{def: X}
    Given a smooth closed 4-manifold $X$, the \textit{equivariant $X$-genus} of a strongly invertible knot $K\subset S^3$, denoted $\tilde{g}_{X}(K)$, is defined as:
    $$\tilde{g}_{X}(K) = \text{min}\{\tilde{g}_{X,\tau}(K) \, | \, \tau \text{ is an involution of $X$}\}$$
\end{definition}

\begin{definition}
    A knot is $(X,\tau)$\textit{-slice}, respectively $\tilde{X}$\textit{-slice}, if $\tilde{g}_{X,\tau}(K)$, respectively $\tilde{g}_{X}(K)$, is equal to 0. We denote the set of all $(X,\tau)$-slice and $\tilde{X}$-slice knots by $\tilde{\mathcal{S}}(X,\tau)$ and $\tilde{\mathcal{S}}(X)$ respectively.
\end{definition}

By definition we get the following chain of inequalities:

$$g_X(K)\leq\tilde{g}_X(K)\leq\tilde{g}_{X,\tau}(K)\leq \tilde{g}_{X,\tau,\Sigma}(K) $$

\noindent which implies:

$$\mathcal{S}(X)\supseteq\tilde{\mathcal{S}}(X)\supseteq \tilde{\mathcal{S}}(X,\tau) \supseteq \tilde{\mathcal{S}}(X,\tau,\Sigma). $$

When $X=S^4$, we get the following example:

\begin{example}\label{eg: fig8 no slice}
    Let $X=S^4$ and $\tau$ be an involution of $S^4$ with two-dimensional fixed point set. Then, fix$(\tau)$ is an $S^2$, possibly knotted. Then we have that $\tilde{X}^\times_{\tau}$ is the 4-ball with involution $\tau$, which may or may not have a knotted fixed point set. Recall that the standard equivariant 4-genus is defined to be the minimum over all extensions of the involution on $S^3$ to $B^4$. This means that we have the following equivalence:
    $$\tilde{g}_4(K) = \tilde{g}_{S^4}(K).$$
    \noindent As mentioned in the introduction, it is not currently known if the minimal genus depends on the extension of $\tau$. Framed in our language of $X$-genera, this is equivalent to the following question:

    \begin{question}
        Does there exist a strongly invertible knot $K$ and an involution $\tau$ on $B^4$ such that $\tilde{g}_{S^4}(K) <\tilde{g}_{S^4,\tau}(K)$, i.e. $\tilde{\mathcal{S}}(S^4)\neq\tilde{\mathcal{S}}(S^4,\tau)$?
    \end{question}
\end{example}

\section{Initial Constructions}

We will now begin to show that there do exist 4-manifolds where the choice of involution affects the slice genera. In particular, we will prove there exists an involution $\tau$ on $S^2\times S^2$ such that $\tilde{\mathcal{S}}(S^2\times S^2)\neq\tilde{\mathcal{S}}(S^2\times S^2,\tau)$. First, we highlight a simple yet effective obstruction to being $(X,\tau,\Sigma)$-slice. To do so, we first highlight a key obstruction:

\begin{lemma}\label{lemma: quotient}
    If $X$ is a smooth 4-manifold with involution $\tau$ such that $\Sigma\subset fix(\tau)$ is a 2-sphere and $K$ is a strongly invertible knot with $g_{X,\tau,\Sigma}=0$, then, letting $Y=X/\tau$,  we have $g_{Y}(K^i)=0$ for $i=0,1$.
\end{lemma}

\begin{proof}
    Let $K$ be a strongly invertible knot in $\tilde{X}^\circ_{\tau,\Sigma}$ and let $D$ be a smooth properly embedded $\tau$-invariant disk bounded by $K$. As $\Sigma$ restricted to $\tilde{X}^\circ_{\tau,\Sigma}$ is a disk, we have $D\cap\Sigma$ is a properly embedded arc $\gamma$, with boundary points $p_1,p_2\in K$. Additionally, recall that $p_1$ and $p_2$ separate the fixed point set of $S^3$ into two half-axis, which we call $h_1$ and $h_2$, such that $\gamma \cup h_i$ bounds a disk  $D_i$ for both $h_i$. Then, considering the quotient $\tilde{X}^\circ_{\tau,\Sigma}/\tau$ we have that $(D\cup D_i)/\tau$, is a properly embedded disk in $Y$ with boundary $K^i$.
\end{proof}

The restriction of this lemma to $S^4$ is a well-known obstruction to being equivariantly slice. In this more general setting, it allows us to immediately obstruct knots from being equivariantly slice in $S^2\times S^2$ paired with a particular involution:

\begin{lemma}\label{lemma: fig8 not slice}
    There exists an involution $\tau$ on $S^2\times S^2$ such that: $$\tilde{\mathcal{S}}(S^2\times S^2,\tau)\subsetneq \mathcal{S}(S^2\times S^2)$$
\end{lemma}

\begin{proof}
    Consider $X=S^2\times S^2$ with the involution $\tau(x,y)=(y,x)$, i.e. the involution which swaps factors. We then have that fix$(\tau)$ is $S^2$, and therefore connected. This means the choice of component of the fixed point is determined and we may consider knots in $\partial \tilde{X}^\circ_\tau$. As mentioned in the introduction, the non-equivariant genus of any knot in $X^\circ$ is always 0, meaning we must have $\tilde{\mathcal{S}}(S^2\times S^2,\tau)\subseteq \mathcal{S}(S^2\times S^2)$. We now show they are not equal.
    Let $K=4_1$, the figure eight knot, with involution shown in Figure \ref{fig:fig8} (left). Considering the quotient, we see that $K^1=T(2,5)$. If $g_{X,\tau}(K)=0$, by Lemma \ref{lemma: quotient}, we would then have that $K^1$ is $\C P^2$-slice, as the quotient of the action on the total space is $\C P^2$. However, its been shown by Yasuhara that this is not the case \cite{yasuhara1992slice}. Thus, $K$ is not in $\tilde{\mathcal{S}}(X,\tau)$.
\end{proof}
\begin{figure}
    \centering
    \includegraphics[width=0.7\linewidth]{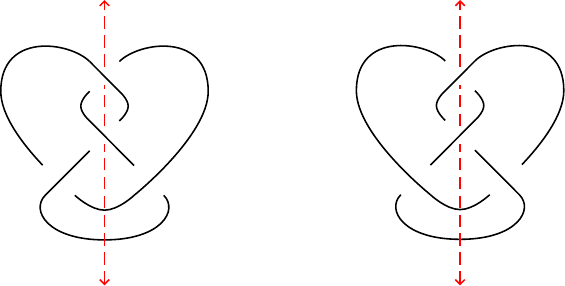}
    \caption{Two different involutions on the figure 8 knot.}
    \label{fig:fig8}
\end{figure}

We will show in Section \ref{sec: examples} that, under a different involution, the figure eight knot is equivariantly slice in $S^2\times S^2$.

\section{The Equivariant Plumbing Construction}

In this section, we will recall Marengon and Mihajlovi\'c's techniques for studying $X$-genus of knots via embeddings of plumbings \cite{marengon2022unknotting}. In recalling their constructions, we will simultaneously extend them to the symmetric setting.

\subsection{An Equivariant Tubing Construction}

The following lemma of \cite{marengon2022unknotting} is a tubing operation which generalizes the ``Norman trick" \cite{norman1969dehn} to more general families of links.

\begin{lemma}[\cite{marengon2022unknotting}, Lemma 3.1] \label{lemma: tubing}
    Let $\Sigma_1,\Sigma_2\subset X$ be two smooth normally immersed surfaces surfaces in a smooth connected 4-manifold $X$. Let $B_1$ and $B_2$ be two disjoint 4-balls with boundaries $S_1$ and $S_2$, respectively, such that $B_1\cap \Sigma_2=B_2\cap \Sigma_1 = \emptyset$. If the links $L_1:=\Sigma_1\cap S_1$ and $L_2:=\Sigma_2\cap S_2$ are mirrors of each other, we can eliminate all the self-intersections of $\Sigma_1\cup \Sigma_2$ in $B_1\cup B_2$ and build a new normally immersed surface $\Sigma\subset X$ by ``tubing", i.e. connecting $\Sigma_1$ and $\Sigma_2$ via disjoint annuli. 
\end{lemma}

We extend this to an equivariant tubing construction for strongly invertible surfaces.

\begin{lemma} \label{lemma: eq tubing}
    Let $\Sigma_1,\Sigma_2\subset X$ be two equivariantly smooth normally immersed surfaces in a smooth connected 4-manifold $X$ with involution $\tau$. Let $B_1$ and $B_2$ be two disjoint $\tau$-invariant 4-balls with boundaries $S_1$ and $S_2$, respectively, such that $B_1\cap \Sigma_2=B_2\cap \Sigma_1 = \emptyset$ and $fix(\tau|_{B_1})$ and $fix(\tau|_{B_2})$ are disks in the same component of $fix(\tau)$. If the links $L_1:=\Sigma_1\cap S_1$ and $L_2:=\Sigma_1\cap S_2$ are mirrors of each-other (as strongly invertible knots), we can eliminate all the self-intersections of $\Sigma_1\cup \Sigma_2$ in $B_1\cup B_2$ and build a new equivariantly normally immersed surface $\Sigma\subset X$ by equivariantly ``tubing", i.e. connecting $\Sigma_1$ and $\Sigma_2$ via disjoint symmetric annuli. 
\end{lemma}

\begin{proof}
    We proceed as in \cite[Proof of Lemma 3.1]{marengon2022unknotting}. Let $\gamma\subset X\backslash (B_1\cup B_2 \cup \Sigma_1 \cup \Sigma_2)$ be an arc connecting $S_1$ and $S_1$. In particular, choose $\gamma$ to be an arc in $fix(\tau)$. Such an arc is guaranteed to exist as we restricted $fix(\tau|_{B_1})$ and $fix(\tau|_{B_2})$ to be disks in the same component of $fix(\tau)$. Then, taking a small equivariant tubular neighborhood of $\gamma$, denoted $\tilde{\mathcal{N}}(\gamma)$, we may construct a new $\tau$-invariant 3-ball $B_3=B_1\cup B_2 \cup \tilde{\mathcal{N}}(\gamma)$. 

    Then, if $L_1$ is equivariantly concordant to $L_2$, there is a symmetric collection of disjoint annuli $A$ in $B_3$ connecting them. The surface obtained by taking $((\Sigma_1\cup\Sigma_2)\backslash (B_1\cup B_2))\cup A$ is then an equivariantly normally immersed surface with the intersections in $B_1$ and $B_2$ removed. 
\end{proof}

In order to utilize this tubing construction, we now turn our attention to specific 4-balls and links appearing via locally partitioned trees, a primary tool for our constructions and results. First, we recall the non-equivariant definitions of \cite{marengon2022unknotting}.

\subsection{Non-Equivariant Locally Bipartitioned Trees}
\begin{definition}
    A \textit{locally bipartitioned tree} $\big( T,\{\Pi_v\}_{v\in V(T)}\big)$ is given by:
    \begin{itemize}
        \item a finite tree $T$ and,
        \item for each vertex $v\in V(T)$, a set $\Pi_v=\{P_v,Q_v\}$ which gives a bipartition of $E(v)$, i.e. $E(V)=P_v\sqcup Q_v$.
    \end{itemize}
    For $v\in V(T)$ and $e\in E(v)$, we let $\pi_v(e)\in\Pi_v$ denote the element of the biparition containing the edge $e$.
\end{definition}

\begin{definition}
    Given a locally biparitioned tree $\big( T,\{\Pi_v\}\big)$, its \textit{associated link} is the unoriented link $L\big( T,\{\Pi_v\})$ in $S^3$ defined in two steps as follows:
    \begin{enumerate}
        \item for each vertex $v\in V(T)$, take a Hopf link with the two components labeled by the two elements of $\Pi_v$;
        \item for each edge $e\in E(T)$, connecting vertices $v$ and $w$, connect sum the two Hopf links associated with $v$ and $w$ at the components labeled $\pi_v(e)$ and $\pi_w(e)$.
    \end{enumerate}
\end{definition}

\begin{definition}
    Let $\big( T,\{\Pi_v\}\big)$ be a locally bipartitioned tree, $\Sigma^2\subset X^4$ be a normally immersed surface, and let $\mathcal{D}(\Sigma)$ denote the set of double points of $\Sigma$. A \textit{suitable embedding} of $\big( T,\{\Pi_v\}\big)$ into $\Sigma$ is an embedding $\funct{f}{T}{\Sigma}$ such that:
    \begin{enumerate}
        \item $f^{-1}(\mathcal{D}(\Sigma))=V(T)$, and
        \item for each vertext $v\in V(T)$, any two edges in $E(v)$ from the same element of the bipartition $\Pi_v$ map into the same local components of $\Sigma$, whereas any two edges from two different elements of $\Pi_v$ map into two different local components of $\Sigma$.
    \end{enumerate}
\end{definition}

\subsection{Equivariantly Locally Bipartitioned Trees}\label{sec: Symmetric Plumbings}

We now define variants of the previous definitions which additionally encode the necessary symmetric information. Due to a variety of restrictions that appear in later constructions, we necessarily add a number of conditions for the equivariant version of locally bipartitioned trees. We will explain the necessity of these additions shortly.


\begin{definition}
    A \textit{$\tau$-equivariantly locally bipartitioned tree} $(T,\{\Pi_v\},\tau,w)$ is a locally bipartitioned tree $(T,\{\Pi_v\})$ with $\tau$ an involution in $Aut(T)$ and $\funct{w}{V(T)}{\{A,B_+,B_-,C\}}$ is a weight function such that for $v \in V(T)$:
    \begin{enumerate}
        \item $\tau$ fixes a single vertex of $T$,
        \item If $w(v)=A$, then $v\neq \tau(v)$, $w(\tau(v))=A$,  and either
        \begin{enumerate}
        \item $\tau(P_v)=P_{\tau(v)}$ and therefore $\tau(Q_v)=Q_{\tau(v)}$ or,
        \item $\tau(P_v)=Q_{\tau(v)}$ and therefore $\tau(Q_v)=P_{\tau(v)}$.
        \end{enumerate}
        \item If $w(v)\in\{B_+,B_-\}$, then $v=\tau(v)$ and $\tau(P_v)=Q_v$.
        \item If $w(v)=C$, then $v=\tau(v)$ and $\tau(P_v)=P_v$.
    \end{enumerate}
\end{definition}

\begin{remark}
    One could define a more general definition of equivariant locally bipartitioned trees which allows for multiple type $C$ vertices. This definition, however, requires additional information to disambiguate future constructions, leading to unnecessary complexity while additionally being unnecessary for producing the results of interest. 
\end{remark}

Since each tree has only a single fixed vertex and that vertex has one of three weights, we get three specific types of trees which can appear. 

\begin{definition}
    For $ \Omega\in \{B_+,B_-,C\}$, a \textit{type} $\Omega$ \textit{locally bipartitioned tree} is a $\tau$-equivariantly locally bipartitioned tree which has a fixed vertex with weight $\Omega$. 
\end{definition}

The weights in this definition come from the labeling convention for types of crossing changes of strongly invertible knots, developed by \cite{boyle2024equivariant}. This relation will be discussed in detail shortly.


Given an equivariantly locally bipartitioned tree, one may construct its associated link as in the non-equivariant setting. While this construction is well-defined in the non-equivariant setting, the same association does not immediately result in a clear choice of strong inversion. This ambiguity is what necessitated the additional properties in the definition of a $\tau$-equivariantly locally bipartioned tree. For example, one can see that are two distinct strongly invertible Hopf links where the involution switches components, as seen in Figure \ref{fig:hopf links}. This is why, in the definition of $\tau$-equivariantly locally bipartitioned trees, we added extra information in the form of weights. Using this added information, we can define the associated strongly invertible link of an equivariantly locally bipartitioned tree.

\begin{definition}
    Given a type $\Omega$ locally bipartitioned tree $(T,\{\Pi_v\},\tau,w)$, its \textit{associated strongly invertible link} has the underlying structure of the associated link of the locally bipartitioned tree $(T,\{\Pi_v\})$, with involution as follows:
    \begin{enumerate}
        \item the Hopf link associated to the fixed vertex is embedded with symmetry determined by its weight as shown in Figure \ref{fig:hopf links},
        \item all other replacements of vertices with Hopf links and connect sums are performed equivariantly. 
    \end{enumerate}
\end{definition}

\begin{figure}
    \centering
    \includegraphics[width=.9\linewidth]{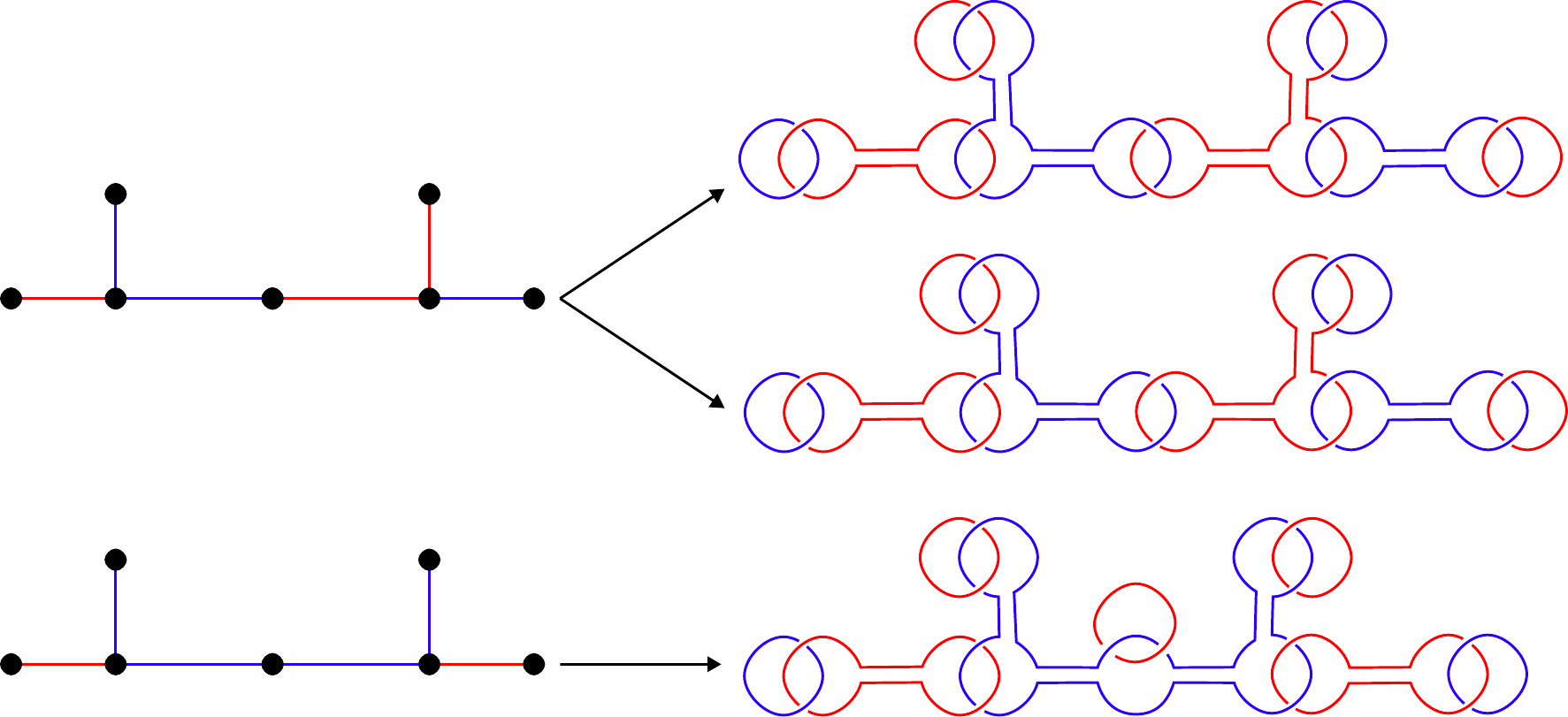}
    
    \caption{Two equivariantly locally biparitioned trees with bipartition given by a coloring of their edges (left) and their associated strongly invertible link (right). On top, the two resolutions refer to choice of $B_+$ (top) or $B_-$ (bottom) for the weight for the fixed vertex.}
    \label{fig:associate links}
\end{figure}

Figure \ref{fig:associate links} shows two associated strongly invertible trees with the same underlying graph structure but distinct symmetric structures coming from the three choices of weight for the fixed vertex.

Our goal now is to construct equivariant embeddings of $\tau$-equivariantly locally bipartitioned trees into immersed surfaces so that a neighborhood of the embedding intersects the surface in the associated strongly invertible link. To do this, we first recall definitions of types of invariant self-intersections. 

In \cite{boyle2024equivariant}, invariant self-intersections are classified as they appear in an equivariant regular homotopy resulting from symmetric crossing changes in the knot diagram. In particular, they define three types of self-intersections.

\begin{definition} Given an equivariant regular homotopy between strongly invertible knots, there are three types of transverse self-intersections which may appear:
\begin{enumerate}
    \item[(1)] A \textit{type A} self-intersection is away from the fixed-point set. Type A intersections come in pairs.
    \item[(2)] A \textit{type B} self-intersection is on the fixed-point set and is the image of two points exchanged by the symmetry. This arises, locally, as the transverse intersection of two disks interchanged by the action.
    \item[(3)] A \textit{type C} self-intersection is on the fixed-point set and is the image of two points fixed by the symmetry. This arises locally as the transverse intersection of two planes each set-wise fixed by the action. 
\end{enumerate}
\end{definition}

We add a further specification here, distinguishing two types of type $B$ self intersections. In particular, a neighborhood of a type $B$ self intersection is given by the transverse intersection of two planes which are swapped by the involution. Assigning an orientation to one plane then, by the symmetry, assigns an orientation to the second plane. If the intersection of the two planes is positive, we say the intersection is type $B_+$ and, similarly, if the intersection is negative, we say it is type $B_-$. 

With this terminology, we are now able to define an appropriate equivariant analog of suitable embeddings:

\begin{definition}
    Let $\tau$ be an involution on a smooth 4-manifold $X^4$, $\Sigma\subset X^4$ be a $\tau$-invariant normally immersed surface with double points $\mathcal{D}(\Sigma)$, and $(T,\{\Pi_v\},\rho,w)$ be a type $\Omega$ locally bipartitioned tree with fixed vertex $v$. A \textit{equivariantly suitable embedding} of $(T,\{\Pi_v\},\rho,w)$ into $\Sigma$ is an equivariant embedding
    $$\funct{f}{(T,\rho)}{(\Sigma,\tau|_\Sigma)}$$
    \noindent such that $f(v)$ is a type $\Omega$ self-intersection and, forgetting symmetry, $f$ is a suitable embedding of $(T,\{\Pi_v\})$ into $\Sigma$.
\end{definition}

By Lemma 3.7 of \cite{marengon2022unknotting}, the link of the equivariantly suitable embedding, i.e. the intersection of a neighborhood of the embedding with the surface, is the associated link (ignoring symmetry). We now show that it is also the associated strongly invertible link.

\begin{lemma}\label{lemma: link from embedding (eq)}
    Given $\Sigma$, a $\tau$-invariant normally immersed surface, and $f$, an equivariantly suitable embedding of an equivariantly locally bipartitioned tree \\$(T,\{\Pi_v\},\rho,w)$ into $\Sigma$, then the link of the embedding of $f$, i.e., the intersection of $\Sigma$ with the boundary of an equivariant neighborhood of $f(T)$, is the associated strongly invertible link of $(T,\{\Pi_v\},\rho,w)$.
\end{lemma}

\begin{proof}
    
In the definition of the associated strongly invertible link, we related to each node weight a specific strongly invertible Hopf link (or pair of Hopf links in the case of weight $A$ vertices). The associated Hopf link to a vertex of weight $\Omega$ is exactly the same as the link obtained from intersecting the boundary of the neighborhood of a type $\Omega$ intersection of an immersed surface with that surface. From here, the result follows exactly as in the non-equivariant setting.

\end{proof}

\subsection{Equivariant Suitable Embeddings in Plumbing Trees and Immersed Surfaces}

The goal now is to be able to equivariantly suitably embed equivariant locally bipartioned trees into a variety of immersed surfaces in a way which allows us to utilize Lemma \ref{lemma: eq tubing}. We begin with plumbing trees $\Sigma=S_1 \cup \cdots \cup S_n$ of spheres in $(X,\tau)$, a smooth 4-manifold with involution $\tau$, which are invariant under $\tau$. As an added restriction, we will consider only plumbing trees as above for which $\tau$ fixes only a single immersed point. We call such trees \textit{simple symmetric plumbing trees}. Note that a non-simple symmetric plumbing tree will always have a subtree which is simple. We may partition the set of all simple symmetric plumbing trees based on the type of self-intersection which appears at the immersed point. These self-intersections are either type $B_+$, $B_-$, or $C$. We then get the following definition:

\begin{definition}
    For $\Omega\in \{B_+,B_-,C\}$, a \textit{type }$\Omega$ \textit{plumbing tree} is a simple symmetric plumbing tree which has a fixed vertex of type $\Omega$.
\end{definition}

In Lemma 3.8 of \cite{marengon2022unknotting}, the authors prove that for any plumbing tree of $n$ spheres (or more generally surfaces), there exists a locally bipartitioned tree with $n-1$ vertices which suitably embeds into the plumbing tree. We now show that if we instead consider a type $\Omega$ plumbing tree with $n$ vertices, we can find an equivariantly locally bipartitioned tree with $n-1$ vertices which equivariantly suitably embeds. 

\begin{lemma}\label{lemma: existence in plumbings (eq)}
    Let $\Sigma=S_1\cup \dots \cup S_n$ be a type $\Omega$ plumbing tree in $(X,\tau)$. Then, there exists a type $\Omega$ locally bipartitioned $(T,\{\Pi_v\},\rho,w)$, with $|V(T)|=n-1$, which equivariantly suitably embeds into $\Sigma$.
\end{lemma}

\begin{proof}
    The existence of a locally bipartitioned tree with $n-1$ nodes which suitably embeds in $\Sigma$ is guaranteed by Lemma 3.8 of \cite{marengon2022unknotting}. The construction used in their proof may be extended directly to $\Omega$ plumbing trees by performing induction on symmetric pairs of nodes connected to the pre-existing tree by equivariant pairs of paths, as opposed to single nodes. This, then, results in an equivariant suitable embedding of $(T,\{\Pi_v\},\rho,w)$ into $\Sigma$.
\end{proof}

Meanwhile, considering immersed surfaces instead of plumbing trees, we able to find large numbers of equivariantly locally bipartitioned trees which can be equivariantly suitably embedded. In the non-equivariant setting, Lemma 3.9 of \cite{marengon2022unknotting} shows that any locally bipartitioned tree with $n$ vertices can be suitably embedded into a normally immersed connected surface with $m>n$ self-intersections. The equivalent results holds in the symmetric setting assuming the self-intersections are of the proper type.

\begin{lemma}\label{lemma: closed surfaces}
    Let $(X,\tau)$ be a smooth 4-manifold with involution $\tau$ and $\Sigma=-\tau(\Sigma)$ be a normally immersed connected surface with $n$ type $A$ self-intersections and at least one type $\Omega$ self-intersection, for $\Omega\in \{B_+,B_-,C\}$. Then any type $\Omega$ locally bipartitioned tree with at most $n+1$ vertices can be equivariantly suitably embedded in $\Sigma$.
\end{lemma}

\begin{proof}
    Let $(T,\{\Pi_v\},\rho,w)$ be an $\Omega$ locally bipartitioned tree with at most $n+1$ vertices. By Lemma 3.9 of \cite{marengon2022unknotting}, $(T,\{\Pi_v\})$ will suitably embed in $\Sigma$. To guarantee that this embedding can be taken to be an equivariant suitable embedding requires that the embedding is equivariant with respect to $\tau$ and that a vertex of weight $\Omega$ is taken to an intersection of type $\Omega$. In the proof of Lemma 3.9 from \cite{marengon2022unknotting} their construction proceeds by induction on the nodes. In the equivariant setting, we can simply begin their inductive argument on the vertex of weight $\Omega$. Since there is a self-intersection of type $\Omega$ in our surface, we may map our single vertex of type $\Omega$ there. From here, their argument holds almost identically if we instead perform induction on pairs of nodes $v$ and $\tau(v)$. The only possible conflict is that we must be sure that the embeddings of the edges connecting $v$ and $\tau(v)$ to the existing tree at nodes $u$ and $\tau(u)$ can be done simultaneously without intersection. If such an intersection existed, we can resolve the intersection equivariantly at the cost of replacing the edges $(v,u)$ and $(\tau(v),\tau(u))$ with $(v,\tau(u))$ and $(\tau(v),u)$, respectively. By symmetry, the tree obtained by this replacement is identical to our original tree. Therefore, we could have avoided having the intersection in the first place. 
\end{proof}

\section{Proofs of Main Results}\label{sec: examples}

In this section, we combine the previous results on equivariant suitable embeddings of locally bipartitioned trees with Lemma \ref{lemma: eq tubing} to prove existence of slice disks and examine various constructions. 

In the non-equivariant setting, one of the central results for constructing slice disks via plumbing trees is the following:

\begin{theorem}[\cite{marengon2022unknotting}, Theorem 1.2]\label{thm: non-eq} If there is a plumbing tree of $n$ smooth spheres $S=S_1\cup \dots \cup S_n$ embedded into a smooth 4-manifold $X^4$, then any knot $K\subset S^3$ with 4-dimensional clasp number $c_4\leq n-1$ is smoothly slice in $X^4$. 
\end{theorem}

We now extend this result to the equivariant setting, proving Theorem \ref{thm: tubing plumbings}:

\begin{theorem1} 
    Let $S=S_1\cup\dots \cup S_n$ be a type $\Omega$ plumbing tree in $(X,\tau)$ with fixed vertex $v\in \Sigma$, a component of $fix(\tau)$. Then a strongly invertible knot $K$ is $(X,\tau,\Sigma)$ slice if it bounds a $\tau$-invariant immersed disk $D$ in $B^4$ satisfying:
    \begin{enumerate}
        \item $D$ has $k\leq n-1$ self-intersections,
        \item $D$ has a single type $-\Omega$ self-intersection,
        \item the remaining $n-2$ self intersections of $D$ are type $A$.
    \end{enumerate}
\end{theorem1}

\begin{proof}
    By Lemma \ref{lemma: existence in plumbings (eq)}, there exists a type $\Omega$ locally bipartitioned tree with $n-1$ nodes which equivariantly suitably embeds in $S$. Then, let $(T,\Pi,\rho,w)$ be an equivariantly locally bipartitioned sub-tree of that tree, such that $T$ has exactly $k$ nodes. This subtree $T$ is obtainable from the original tree by pruning paired nodes of type $A$ until you get a subtree containing $k$ nodes. Then $(T,\Pi,\rho,w)$ is also a type $\Omega$ tree which equivariantly suitably embeds in $S$. Additionally, by Lemma \ref{lemma: closed surfaces}, the mirror of this tree also equivariantly suitably embeds into $D$.

    Now, instead of a disk in $B^4$, $D$ can be considered as a disk in $S^3\times I$ by removing a neighborhood of the fixed-point set away from $D$. Therefore, we may construct an equivalent disk, which we also call $D$, in a symmetric collar neighborhood of $\partial X^\circ_{\tau,\Sigma}$ bounded by $K\subset \partial X^\circ_{\tau,\Sigma}$. Then, $D$ is a symmetric disk in $X^\circ_{\Sigma,\tau}$, bounded by $K$, which $(T,\Pi,\rho,w)$ still equivariantly suitably embeds into. Since this is in a small neighborhood of the boundary, we may take $D$ to be disjoint from the plumbing forest $S$.    

    By Lemma \ref{lemma: link from embedding (eq)}, the links in the boundaries of small neighborhoods of the embeddings of $T$ into $D$ and $S$ are the same as strongly invertible links and isotopic to their own mirror images. Thus, since the neighborhoods of each tree have fixed point set a disk in the same components $\Sigma$ of $\tau$, we may apply Lemma \ref{lemma: eq tubing}. This then equivariantly tubes them together, removing all self-intersections in $T$. As in the non-equivariant case, this adds no genus and therefore results in a $\tau$-invariant disk, i.e. a symmetric slice disk for $K$.
    

\end{proof}

Applying this result, and Lemma \ref{lemma: fig8 not slice}, we get the proof of Theorem \ref{thm: different S2 sliceness}:

\begin{proof}[Proof of Theorem \ref{thm: different S2 sliceness}]
    Consider $S^2\times S^2$ with the involution $\tau_1$ swapping the two spheres. That is to say, letting $(x,y)$ be a point in $S^2\times S^2$, $\tau_1(x,y)=(y,x)$. This is the same involution used in the proof of Lemma \ref{lemma: fig8 not slice} where we saw that the figure eight knot with a specific strong inversion was not equivariantly slice in $(S^2\times S^2,\tau_1)$. We will see here that the mirror of this knot is $(S^2\times S^2,\tau_1)$-slice. We will let $K$ refer to the original strongly invertible figure eight and $mK$ refer to its mirror.
    
    A Kirby diagram for $(S^2\times S^2,\tau_1)$ is shown in Figure \ref{fig:hopf links} (middle left) by realizing the components of the Hopf link as the attaching spheres for 0-framed 2-handles. The linking of the Hopf link in the diagram with the fixed axis is important. For this Hopf link, $(S^2\times S^2,\tau_1)$ has a type $B_+$ plumbing tree, consisting of a single plumbing of two spheres, $S^2\times \{pt\}$ and $\{pt\}\times S^2$. Since $mK$ can be equivariantly unknotted via a single type $B_+$ crossing change, this means that it bounds an invariant immersed disk in $B^4$ with a single clasp which is of type $B_+$. Therefore, by Theorem \ref{thm: tubing plumbings}, we have that $mK$ is equivariantly slice in $(S^2\times S^2,\tau_1)$.

    Now consider $S^2\times S^2$ with the involution $\tau_2$, which both swaps and reflects the factors. That is to say, letting $\funct{r}{S^2}{S^2}$ be reflection across a hemisphere, $\tau_2(x,y)=\tau_1(r(x),r(y))$. Now there exists a type $B_-$ plumbing tree, making $K$ equivariantly slice by the same argument as above. By a similar argument to Lemma \ref{lemma: fig8 not slice}, we also see that $mK$ is, in fact, not equivariantly slice in $(S^2\times S^2,\tau_2)$. 

    Thus, we must then have that $\tS(S^2\times S^2,\tau_1)$ and $\tS(S^2\times S^2,\tau_2)$ are distinct, with neither contained in the other.
\end{proof}

\begin{figure}[ht]
        \centering
        \includegraphics[width=0.55\linewidth]{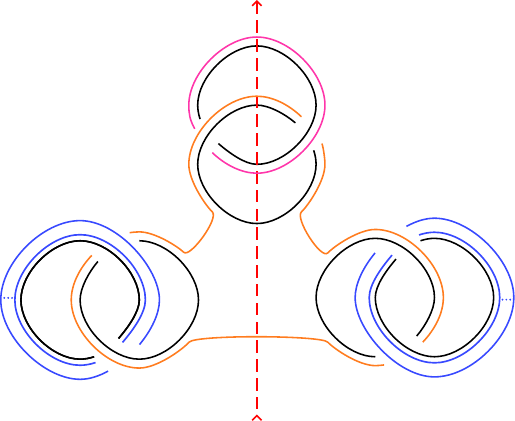}
        \caption{Kirby diagram for $(S^2\times S^2)^3$ with a type $C$ plumbing tree.}
        \label{fig:s2xs2}
    \end{figure}

We can also use Theorem \ref{thm: tubing plumbings} to prove Corollary \ref{cor: S^2timesS^2}.

\begin{proof}
    Let $\tau$ be the orientation preserving involution on $X=(S^2\times S^2)^3$ which interchanges the first two $S^2\times S^2$ summands, fixing some $\Sigma=S^2$ in the boundary of the 4-ball removed from each component to form the connect sum, and acts on the third via reflection on each factor (see the Kirby diagram in Figure \ref{fig:s2xs2}). Then, for every $n>0$ there exists a type $C$ pluming tree $S_n$ in $(X,\tau)$ consisting of $2n+2$ spheres which has a single type $C$ fixed point and $2n$ type $A$ self-intersections. This tree is depicted in Figure \ref{fig:s2xs2}, where the black Hopf links are 0-framed 2-handles and the colored circles are cross-sections of the spheres in the plumbing tree.

    Now, let $K$ be a strongly invertible knot and let $K'$ be the knot obtained by performing a single type $C$ crossing change on $K$. By \cite{boyle2024equivariant}, $K'$ can be equivariantly unknotted via $n$ type $A$ moves, for some finite $n\geq 0$. Therefore, $K$ bounds a disk $D$ with $2n+1$ self-intersections, a single type $C$ intersection, and $2n$ type $A$ self-intersections. Thus, by Theorem \ref{thm: tubing plumbings}, $K$ is $(X,\tau)$-slice.
\end{proof}

\printbibliography

\end{document}